\documentclass[12pt]{amsart}
\usepackage[margin=1in]{geometry}

\newcommand{\Q}{\mathbb{Q}}
\newcommand{\Z}{\mathbb{Z}}
\newcommand{\F}{\mathbb{F}}

\newcommand{\C}{\mathbb{C}}
\newcommand{\Gal}{{\rm Gal}}
\newcommand{\GL}{{\rm GL}}

\newcommand{\AGL}{{\rm AGL}}

\newcommand{\Row}{{\rm Row}}

\usepackage{amsthm,amssymb,wasysym,graphicx,url}
\newtheorem{thm}{Theorem}
\newtheorem{lem}[thm]{Lemma}
\newtheorem{clm}[thm]{Claim}
\newtheorem*{defn}{Definition}
\newtheorem*{ack}{Acknowledgements}
\newcommand{\im}{{\rm im}~}

\newcommand{\ghappy}{\substack{\smiley{}\hspace{0.05in}\vspace{-0.05in}\\\vspace{0.05in}\displaystyle{\Gamma_{0}(2)}}}

\begin{document}
\title[Odd Order Reductions]{The density of odd order reductions for elliptic curves with a rational point of order $2$}
\author{Ke Liang}
\author{Jeremy Rouse}
\begin{abstract}
  Suppose that $E/\Q$ is an elliptic curve with a rational point $T$
  of order $2$ and $\alpha \in E(\Q)$ is a point of infinite order. We
  consider the problem of determining the density of primes $p$ for
  which $\alpha \in E(\F_{p})$ has odd order. This density is
  determined by the image of the arboreal Galois representation
  $\tau_{E,2^{k}} : \Gal(\overline{\Q}/\Q) \to \AGL_{2}(\Z/2^{k}
  \Z)$. Assuming that $\alpha$ is primitive (that is, neither $\alpha$
  nor $\alpha + T$ is twice a point over $\Q$) and that the image of
  the ordinary mod $2^{k}$ Galois representation is as large as
  possible (subject to $E$ having a rational point of order $2$), we
  determine that there are $63$ possibilities for the image of
  $\tau_{E,2^{k}}$. As a consequence, the density of primes $p$ for
  which the order of $\alpha$ is odd is between $1/14$ and $89/168$.
\end{abstract}

\maketitle

\section{Introduction and Statement of Results}

Let $E/\Q$ be an elliptic curve and $\alpha \in E(\Q)$ be a point of infinite
order. For each prime $p$ for which $E/\F_{p}$ has good reduction, the
reduced point $\overline{\alpha} \in E(\F_{p})$ has finite order, say $M_{p}$.
What can one say about the sequence of numbers $M_{p}$ as $p$ varies?

We are particularly interested in how often the number $M_{p}$ is
odd. More precisely, let $S$ denote the set of primes $p$ for which $p
\nmid N(E)$ and for which $M_{p}$ is odd. We are interested in the
relative density of $S$ within the set of primes, namely $\lim_{x \to
  \infty} \frac{\pi_{S}(x)}{\pi(x)}$.  Naively, one would expect this
to occur 50 percent of the time, but this is often not the case. For
example, in \cite{Jones}, the authors show that if $E : y^{2} + y =
x^{3} - x$ and $\alpha = (0,0)$, then
\[
\lim_{x \to \infty} \frac{\pi_{S}(x)}{\pi(x)} = \frac{11}{21} \approx 0.52381.
\]
Moreover, for each positive integer $k$, there is a Galois
representation (depending both on $E$ and the point $\alpha$)
$\tau_{E,2^{k}} : \Gal(\overline{\Q}/\Q) \to (\Z/2^{k}\Z)^{2} \rtimes
\GL_{2}(\Z/2^{k} \Z) := \AGL_{2}(\Z/2^{k} \Z)$, and Theorem 3.2 of
\cite{Jones} implies that the relative density of the set $S$ only
depends on the image of $\tau_{E,2^{k}}$ (for all $k$).

Variations of this problem have been studied by several authors. (See for
example \cite{Pink}, \cite{PeruccaIMRN}, \cite{LombardoPerucca}, and \cite{DKR}.) In \cite{2015REU}, the authors study the problem of determining the image of $\tau$ subject to the constraints that (i) the usual mod $2^{k}$
Galois representation $\rho_{E,2^{k}} : \Gal(\overline{\Q}/\Q) \to \GL_{2}(\Z/2^{k} \Z)$ is surjective for all $k$, and (ii) $\alpha \in E(\Q)$ is not equal to $2 \gamma$ for any $\gamma \in E(\Q)$. They show that under these hypotheses, there are two possibilities for the image of $\tau$: $\AGL_{2}(\Z/2^{k} \Z)$, and an index $4$ subgroup thereof. In the latter case, the density is $179/336$.

The goal of the present paper is to study the situation when the image of
$\rho_{E,2^{k}}$ is
\[
\Gamma_{0}(2) := \left\{ \begin{bmatrix} a & b \\ c & d \end{bmatrix} \in \GL_{2}(\Z/2^{k} \Z) : c \equiv 0 \pmod{2} \right\}.
\]
In this situation, $E$ has a rational point of order $2$, which we refer to
as $T$. We wish to impose the ``primitivity condition'' that neither $\alpha$
nor $\alpha+T$ is equal to $2 \gamma$ for some $\gamma \in E(\Q)$. This condition is automatically satisfied if we choose $\alpha$ to be a generator of the
group $E(\Q)$. This assumption means that there are only finitely many
possibilities for the image of $\tau_{E,2^{k}}$.

\begin{thm}
\label{main}
Let $E/\Q$ be an elliptic curve with the image of $\rho_{E,2^{k}}$ equal
to $\Gamma_{0}(2)$ for all $k$. Let $T$ denote the unique rational point of
order $2$ in $E(\Q)$, and suppose that $\alpha \in E(\Q)$ has the property
that neither $\alpha$ nor $\alpha+T$ is of the form $2 \gamma$ for $\gamma \in E(\Q)$. Then, there are precisely $63$ possibilities for the image of
$\tau_{E,2^{k}}$ (up to conjugacy in $\AGL_{2}(\Z/2^{k} \Z)$), and there
are $21$ possibilities for the relative density of the set of primes
for which $\alpha$ has odd order modulo $p$, ranging from $1/14 \approx 0.0714$ to $89/168 \approx 0.5298$.
\end{thm}

We can translate all of the hypotheses in the theorem into purely group theoretic statements. In Section~\ref{grouptheory} we solve the problem of determining
which possible subgroups of $\AGL_{2}(\Z/2^{k} \Z)$ could be the image of
$\tau$ by computing with Magma. In all, we find that there are 63 candidates
as shown in the graph below.

\begin{center}
\begin{figure}[ht]
\includegraphics[width=\textwidth]{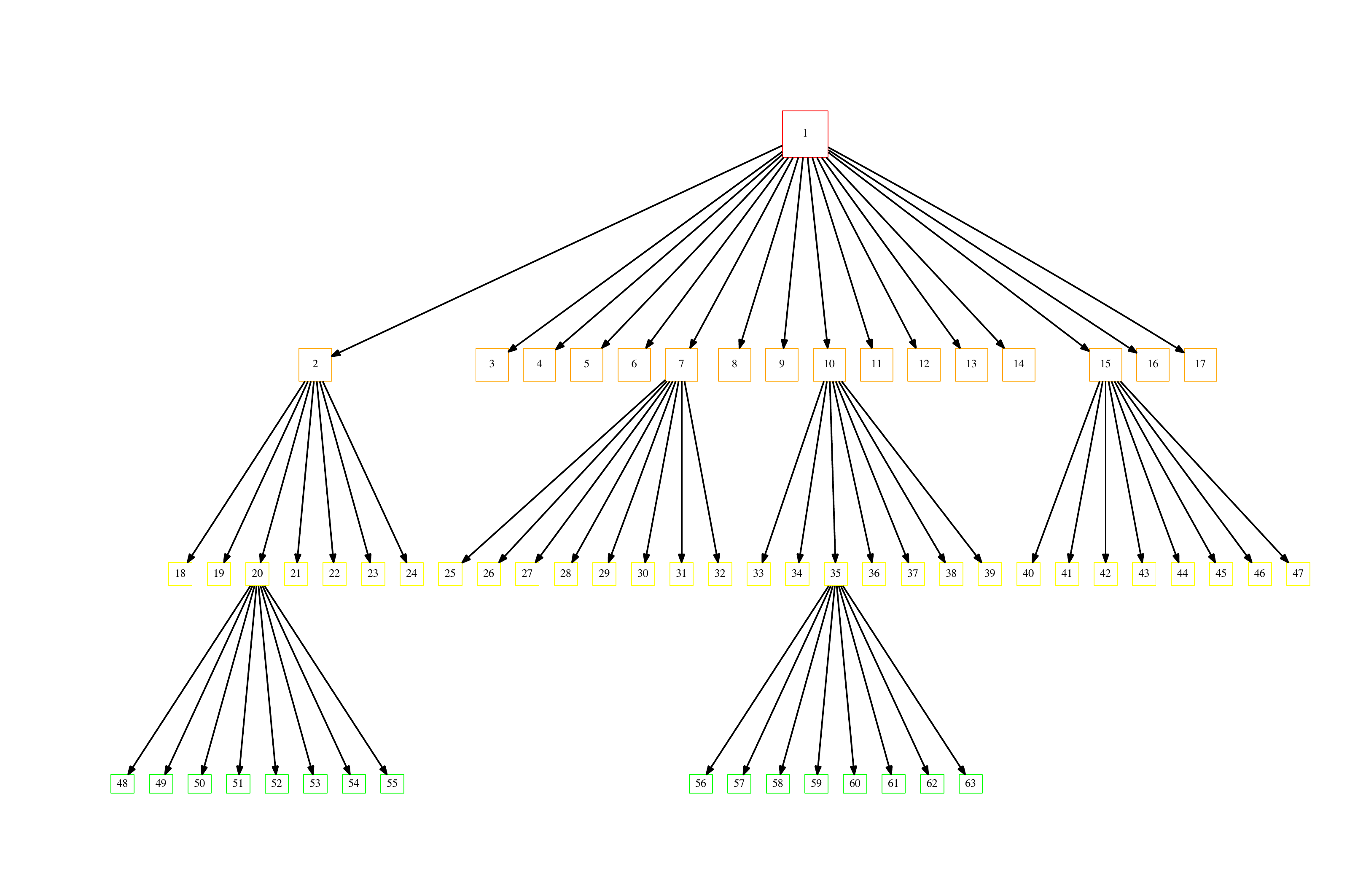}
\end{figure}
\end{center}

In Section~\ref{interpret} we give concrete conditions that describe
when the image of $\tau_{E,2^{k}}$ lies in each of the subgroups
above, and we exhibit a pair $(E,\alpha)$ of an elliptic curve and a
point $\alpha$ with each of the $63$ possible images.  Finally, in
Section~\ref{calc_density} we give a method suited to computer
computation for computing the relative density of the set of primes
where $\alpha$ has odd order, given the image of $\tau_{E,2^{k}}$.

\begin{ack}
The authors used Magma \cite{Magma} (version 2.23-11) for computations. A part of this work constitutes the first author's thesis at Wake Forest University.
\end{ack}

\section{Background}

Let $E/\Q$ be an elliptic curve. For a positive integer $m$, define
$E[m] = \{ Q : mQ = 0 \}$, where $0$ denotes the identity of the group
law. We have that $E[m](\C) \cong (\Z/m\Z)^{2}$. We say that $E$ has
good reduction at $p$ if $E/\F_{p}$ is non-singular, and we say that
$E$ has bad reduction otherwise. If we choose a basis $\langle A, B
\rangle$ for $E[m]$, we get the usual mod $m$ Galois representation
$\rho_{E,m} : \Gal(\Q(E[m])/\Q) \to \GL_{2}(\Z/m\Z)$ given by
$\rho_{E,m}(\sigma) = \begin{bmatrix} a & b \\ c & d \end{bmatrix}$,
where
\[
\sigma(A) = aA + bB \text{ and } \sigma(B) = cA + dB.
\]

If $E/\Q$ is an elliptic curve and $\im \rho_{E,2^{k}} = \Gamma_{0}(2)$,
then $E$ has a rational point of order $2$. By translating this rational
point to the origin, we can assume that $E : y^{2} = x^{3} + ax^{2} + bx$
for some $a, b \in \Z$ and that $T = (0,0)$ is the unique point of order $2$.
The subgroup $\{ (0 : 1 : 0), T \}$ is the kernel of an isogeny
$\phi : E \to E'$ where $E' : y^{2} = x^{3} + (-2a)x^{2} + (a^{2} - 4b)x$,
and
\[
\phi(x,y) = \left(\frac{y^{2}}{x^{2}}, \frac{y(x^{2}-b)}{x^{2}}\right).
\]
The dual isogeny $\psi : E' \to E$ has a similar formula (see \cite{SilvermanTate} page 79). We have $\psi \circ \phi(P) = 2P$ for all $P \in E(\Q)$ and
$\phi \circ \psi(Q) = 2Q$ for all $Q \in E'(\Q)$.

Let $\alpha \in E(\Q)$ be a point of infinite order. For a positive
integer $m$, let $[m^{-1}]\alpha = \{ P \in E(\C) : mP = \alpha \}$ be
the preimage of $\alpha$ under the multiplication by $m$ map. For a
positive integer $k$, there are $4^{k}$ points $\delta \in
[(2^{k})^{-1}] \alpha$.  The coordinates of these points are algebraic
numbers and we define $K_{k}$ to be the field obtained by adjoining to
$\Q$ these coordinates. The extension $K_{k}/\Q$ is Galois. For each
positive integer $k$, we choose one point $\beta_{k}$ so that $2^{k}
\beta_{k} = \alpha$ and we choose these in such a way that $2
\beta_{k} = \beta_{k-1}$. If $\sigma \in \Gal(K_{k}/\Q)$, 
$\sigma(\beta_{k})$ is also in $[(2^{k})^{-1}] \alpha$ and hence
we can write
\[
\sigma(\beta_{k}) = \beta_{k} + eA + fB,
\]
where $\langle A, B \rangle = E[2^{k}]$. The map $\tau_{E,2^{k}} : \Gal(K_{k}/\Q) \to \AGL_{2}(\Z/2^{k} \Z)$ given by
\[
\tau_{E,2^{k}}(\sigma) = \left(\begin{bmatrix} e & f \end{bmatrix},
  \rho_{E,2^{k}}(\sigma)\right)
\]
is the arboreal Galois representation. Here $\rho_{E,2^{k}}$ is as
given above. We think of $\begin{bmatrix} e & f \end{bmatrix} \in
(\Z/2^{k} \Z)^{2}$ as being a row vector. The group operation
on $\AGL_{2}(\Z/2^{k} \Z)$ is given by $(\vec{v}_{1},M_{1}) * (\vec{v}_{2},M_{2})
= (\vec{v}_{2} + \vec{v}_{1} M_{1}, M_{1} M_{2})$. This makes it
so that the action of $\sigma \in \Gal(K_{k}/\Q)$ on $[(2^{k})^{-1} \alpha]$
is equivalent to the action of the function $\vec{v} + \vec{x} M$ on
row vectors $\vec{x} \in (\Z/2^{k} \Z)^{2}$ (where $(\vec{v},M) = \tau_{E,2^{k}}(\sigma)$). Moreover, the map
\[
\left(\begin{bmatrix} e & f \end{bmatrix}, \begin{bmatrix} a & b \\ c & d \end{bmatrix} \right) \mapsto \begin{bmatrix}
  a & b & 0\\
  c & d & 0 \\
  e & f & 1 \end{bmatrix}
\]
from $\AGL_{2}(\Z/2^{k} \Z)$ to $\GL_{3}(\Z/2^{k} \Z)$ is a homomorphism, and
it is convenient to think of $\AGL_{2}(\Z/2^{k} \Z)$ as a subgroup
of $\GL_{3}(\Z/2^{k} \Z)$ for computational purposes. Note
that $\AGL_{2}(\Z/2^{k} \Z)$ consists of pairs of vectors in $(\Z/2^{k} \Z)^{2}$
and matrices $M \in \GL_{2}(\Z/2^{k} \Z)$ and so $|\AGL_{2}(\Z/2^{k} \Z)| =
2^{2k} \cdot |\GL_{2}(\Z/2^{k} \Z)| = 3 \cdot 2^{6k-3}$.

We wish to give a brief accounting of the connection between the
image of $\tau_{E,2^{k}}$ and the relative density of primes $p$ for which
$\alpha \in E(\F_{p})$ has odd order. For a more detailed explanation,
see Theorem 5.1 of \cite{2015REU}.

Let $p$ be an odd prime with $p \nmid N(E)$. Then it is
straightforward to see that $\alpha \in E(\F_{p})$ has odd order if
and only if for all $k \geq 1$, there is some point $\beta_{k} \in
[(2^{k})^{-1}]\alpha \cap E(\F_{p})$. This corresponds to some element
of $[(2^{k})^{-1}] \alpha$ ``having coordinates in $\F_{p}$''. More
precisely, if $\mathcal{O}_{K_{k}}$ is the ring of algebraic integers
in $K_{k}$, and $\mathfrak{p}$ is a prime ideal above $p$ and
$\sigma_{\mathfrak{p}} \in \Gal(K_{k}/\Q)$ is the corresponding
Frobenius automorphism at $\mathfrak{p}$, saying that some
element of $[(2^{k})^{-1}] \alpha$ ``has coordinates in $\F_{p}$''
is asking for $\sigma_{\mathfrak{p}}$ to fix some element of $[(2^{k})^{-1}] \alpha$. Since the action of $\sigma_{\mathfrak{p}}$ on $[(2^{k})^{-1}] \alpha$ is
equivalent to the action of the function $f(\vec{x}) = \vec{v} + \vec{x} M$
on $(\Z/2^{k} \Z)^{2}$, where $(\vec{v},M) = \tau_{E,2^{k}}(\sigma_{\mathfrak{p}})$,
this is the same as asking for $f(\vec{x})$ to have a fixed point. Equivalently,
$\vec{v} \in \Row(M-I)$. In the case that $\det(M-I) \not\equiv 0 \pmod{2^{k}}$,
$\vec{v} \in \Row(M-I)$ implies that $[(2^{k})^{-1}] \alpha \cap E(\F_{p})$
is nonempty for all $k$ and so $\alpha \in E(\F_{p})$ has odd order.
The relative density of $p$ for which $\det (\rho_{E,2^{k}}(\sigma_{\mathfrak{p}})-I) \equiv 0 \pmod{2^{k}}$ tends to zero as $k \to \infty$. Thus, the Chebotarev
Density theorem implies that
\begin{equation}
\label{densityeq}
\lim_{x \to \infty} \frac{\pi_{S}(x)}{\pi(x)}
= \lim_{k \to \infty} \frac{\# \{ (\vec{v},M) \in \im~\tau_{E,2^{k}} : \vec{v} \in \Row(M-I)\}}{|\im~\tau_{E,2^{k}}|}.
\end{equation}

\section{Group Theory}
\label{grouptheory}
We study the possibilities for the image of
$\tau$, assuming $\im \rho$ is as large as possible subject to $E$
having a rational point of order $2$. Specifically, we will deal with
the case when the image of $\rho_{E,2^k}$ is $\Gamma_{0}(2)$. In
essence, we have a homomorphism $\tau:\Gal(K_k)/\Q\to G$, where $G$ is
some finite group, and we are interested in understanding when $\tau$
is surjective.

\begin{defn}
  Let $\ghappy$ be the subgroup of $\AGL_2(\Z/2^k\Z)$ consisting the matrices
\[\begin{bmatrix}a&b&0\\c&d&0\\e&f&1\end{bmatrix}\]
with $e,f$ any entries in $\Z/2^k\Z$ and $\begin{bmatrix} a&b\\c&d \end{bmatrix}$ is in $\Gamma_0(2)$.
\end{defn}
We have $[\AGL_2(\Z/2^k\Z):\ghappy]=3$, so $|\ghappy|=2^{6k-3}$.

Recall that if $G$ is a finite group, the Frattini subgroup $\Phi(G)$ is the intersection of all maximal subgroups of $G$. Here are two facts about the Frattini subgroup.

\begin{thm}[Theorem $5.2.13$ of {\cite[p.~135]{Robinson}}]
\label{normal}
If $G$ is a finite group and $N$ is a normal subgroup of $G$, then $\Phi(N)\trianglelefteq \Phi(G)$.
\end{thm}

\begin{thm}[Theorem $5.3.1$ of {\cite[p.~139]{Robinson}}]
\label{square}
Let $G$ be a group whose order is a power of $2$. If $M \subseteq G$ is a maximal subgroup, then $|G:M|=2$. As a consequence, if $x\in G$, then $x^2 \in \Phi(G)$.
\end{thm}

The next result gives us a method for finding all maximal subgroups
of a subgroup $M \subseteq \ghappy$ that contains the kernel of reduction mod $8$.

\begin{thm}
For any $k\geq 3$, if $M\subseteq \ghappy$ and $M$ contains $\{(\vec{v},M):\vec{v}\equiv 0 \pmod{8}, M \equiv I \pmod{8}\}$, then $\Phi(M) \supseteq \{(\vec{v},M):\vec{v}\equiv 0 \pmod{16}, M \equiv I \pmod{16}\}$.
\end{thm}
\begin{proof}
If $N=\{(\vec{v},M):\vec{v}\equiv 0\pmod{8}, M\equiv I \pmod{8}\}$, then $N$ is the kernel of the map $\phi: G\to \AGL_2(\Z/8\Z)$ with $\phi((\vec{v},M))=(\vec{v} \pmod 8, M \pmod 8)$, so $N$ is a normal subgroup of $G$. Since $|N|=2^{6k-20}$ so that $N$ is a power of 2, we have that if $x\in N$, then $x^2\in \Phi(N)$ by Theorem~\ref{square}. If $\vec{v}\equiv 0 \pmod{16}$, then $\vec{v}=2\vec{w}$ with $x=(\vec{w},I)\in N$. So, we have $x^2=(2\vec{w},I)=(\vec{v},I)\in \Phi(N)$.

We claim that every element $(\vec{0},I+2^l g)$ is in $\Phi(N)$ for $4\leq l\leq k$, where $g$ is any $2\times 2$ matrix with entries in $\Z/2^{k-l}\Z$.

We prove this by backwards induction on $l$. The case that $l=k$ is the base case and is trivial. Assume that for all $l>r$, every element $(\vec{0},I+2^l g)$ is in $\Phi(N)$. We prove the same for $l=r$. We have $x=(\vec{0}, I+2^{r-1}g)\in N$ and so $x^2=(\vec{0}, (I+2^{r-1}g)^2)=(\vec{0}, I+2^r g+2^{2r-2}g^2)\in \Phi(N)$. So we have
\[I+2^r g+2^{2r-2}g^2\equiv I + 2^r g \pmod{2^{2r-2}}.\]
Since $r\geq 4$, we have that $2r-4\geq r$ and by the induction hypothesis, every matrix $h$ congruent to the identity mod ${2^{2r-2}}$ has $(\vec{0},h)\in \Phi(N)$. This completes the induction.	

It follows that $\Phi(N)$ contains all the elements of the form $(0,I+2^4g)$ and all elements of the form $(\vec{v}, I)$ where $\vec{v}\equiv 0 \pmod{16}$. Thus, it contains all elements of the form
\[(0,I+2^4g)*(\vec{v},I)=(\vec{v}+\vec{0}I,I+2^4g)=(\vec{v}, I+2^4g).\]
This shows that $\Phi(N)$ contains all of these elements and by Theorem~\ref{normal}, $\Phi(N)\subseteq \Phi(M)$. 
\end{proof}

Here are our computation steps in Magma, from which we can know there are only $63$ possible candidates:
\begin{enumerate}
\item We start with $\ghappy$ and compute its maximal subgroups. We say a subgroup $M$ is \emph{happy} if the image of $M$ inside $\GL_2(\Z/2^k\Z)$ is $\Gamma_0(2)$, $M\not\subseteq\left\{\begin{bmatrix}a&b&0\\c&d&0\\e&f&1\end{bmatrix}:e\equiv f\equiv 0 \pmod{2}\right\}$ $(\alpha\not\in 2E(\Q))$ and $M\not\subseteq\left\{\begin{bmatrix}a&b&0\\c&d&0\\e&f&1\end{bmatrix}:e \equiv c/2 \pmod{2}, f \equiv (d-1)/2 \pmod{2}\right\}$ $(\alpha+T\not\in 2E(\Q))$.
\item We make a list of all happy subgroups $\subseteq \AGL_2(\Z/8\Z)$ and we get $63$ subgroups (up to conjugacy).
\item We make a list of all happy subgroups $\subseteq \AGL_2(\Z/16\Z)$ and we get $63$ subgroups (up to conjugacy).
\end{enumerate}
If $M$ is one of these $63$ happy subgroups of $\AGL_2(\Z/8\Z)$ and $K\subseteq M$ is any happy subgroup of it, then $K\subseteq L$ for some maximal subgroup $L$ of $M$. Then, $\Phi(M)\subseteq L$ and $\{(\vec{v},M):\vec{v}\equiv 0 \pmod{16}, M\equiv I \pmod{16}\}\subseteq \Phi(M)$ so that $L$ will show up in our list in step $3$. In other words, $H$ is one of our $63$ subgroups if and only if \[H\supseteq \left\{\begin{bmatrix}a&b&0\\c&d&0\\e&f&1\end{bmatrix}:a\equiv d \equiv 1 \pmod{8}, b\equiv c \equiv e \equiv f \equiv 0\pmod{8}\right\}.\]

For use in the next section, it will be helpful to define certain subgroups and discuss the corresponding subfields.

\begin{defn}
Let $G$ be a finite group and suppose that $N_1$ and $N_2$ are subgroups of $G$ with $|G:N_{1}|=|G:N_{2}|=2$. It follows that $N_1$ and $N_2$ are both normal subgroups. Define $N_1*N_2=\{g\in G:$ either $g\in N_1$ and $g\in N_2$ or $g\not \in N_1$ and $g\not \in N_2\}$.
\end{defn}
\begin{clm}
$N_1*N_2$ is a subgroup of $G$ and $|G:N_1*N_2|=2$.
\end{clm}
\begin{proof}
Let $\phi: G\to (G/N_1)\times(G/N_2)$ be given by $\phi(x)=(xN_1,xN_2)$. Then, $G/N_1\cong \Z/2\Z$ and $G/N_2\cong \Z/2\Z$. Then, $N_1*N_2=\{g\in G: \phi(g)=(0,0)$ or $\phi(g)=(1,1)\}$ is clearly a subgroup of $G$ of index $2$.
\end{proof}
Suppose $M\subseteq \ghappy$ is the pre-image under $\ghappy \to \Gamma_0(2)$ of a maximal subgroup of $\Gamma_0(2)$. For each $H\subseteq \Gamma_{0}^{\smiley}(2)$ that is happy, with ${\rm im}~ \rho = \Gamma_0(2)$, $M*H$ is a happy subgroup of $\Gamma_{0}^{\smiley}(2)$ with ${\rm im}~ \rho =\Gamma_0(2)$.

If $G=\Gal(K/\Q)$ and $N_1$ and $N_2$ are subgroups of $G$, they correspond to subfields $K_1/\Q$ and $K_2/\Q$. Since $|G:N_1|=2$ and $|G:N_2|=2$, $|K_1:\Q|=|K_2:\Q|=2$. Thus, $K_1=\Q(\sqrt{d_1})$ and $K_2=\Q(\sqrt{d_2})$, where $N_1=\Gal(K/K_1)$ and $N_2=\Gal(K/K_2)$. In this case, $N_1*N_2$ corresponds to $\Q(\sqrt{d_1d_2})$. This is because we have $\sigma(\sqrt{d_1})=\sqrt{d_1}$ if and only if $\sigma\in N_1$; otherwise $\sigma(\sqrt{d_1})=-\sqrt{d_1}$ and $\sigma(\sqrt{d_1})=\sqrt{d_2}$ if and only of $\sigma\in N_2$.

Then, we have that $\sigma(\sqrt{d_1d_2})=\sqrt{d_1d_2}$ if and only if 
\[\frac{\sigma(\sqrt{d_1})}{\sqrt{d_1}}=\frac{\sigma(\sqrt{d_2})}{\sqrt{d_2}},\]
and so $\sigma(\sqrt{d_1 d_2})=\sqrt{d_1d_2}$ if and only if $\sigma\in N_1*N_2$.

\section{Interpretation of the images of $\tau$}
\label{interpret}
In this section, we describe the interpretation of the 63 possible different
images of $\tau_{E,2^{k}}$, up to conjugacy. We will describe the methods
we used to find an interpretation for each of the possible images of
$\tau_{E,2^{k}}$, and end this section by giving a table of images,
interpretations, densities, and elliptic curves that yield
each image.

The lattice of the happy subgroups of $\ghappy$ is the following.

\begin{center}
\begin{figure}[ht]
\includegraphics[width=\textwidth]{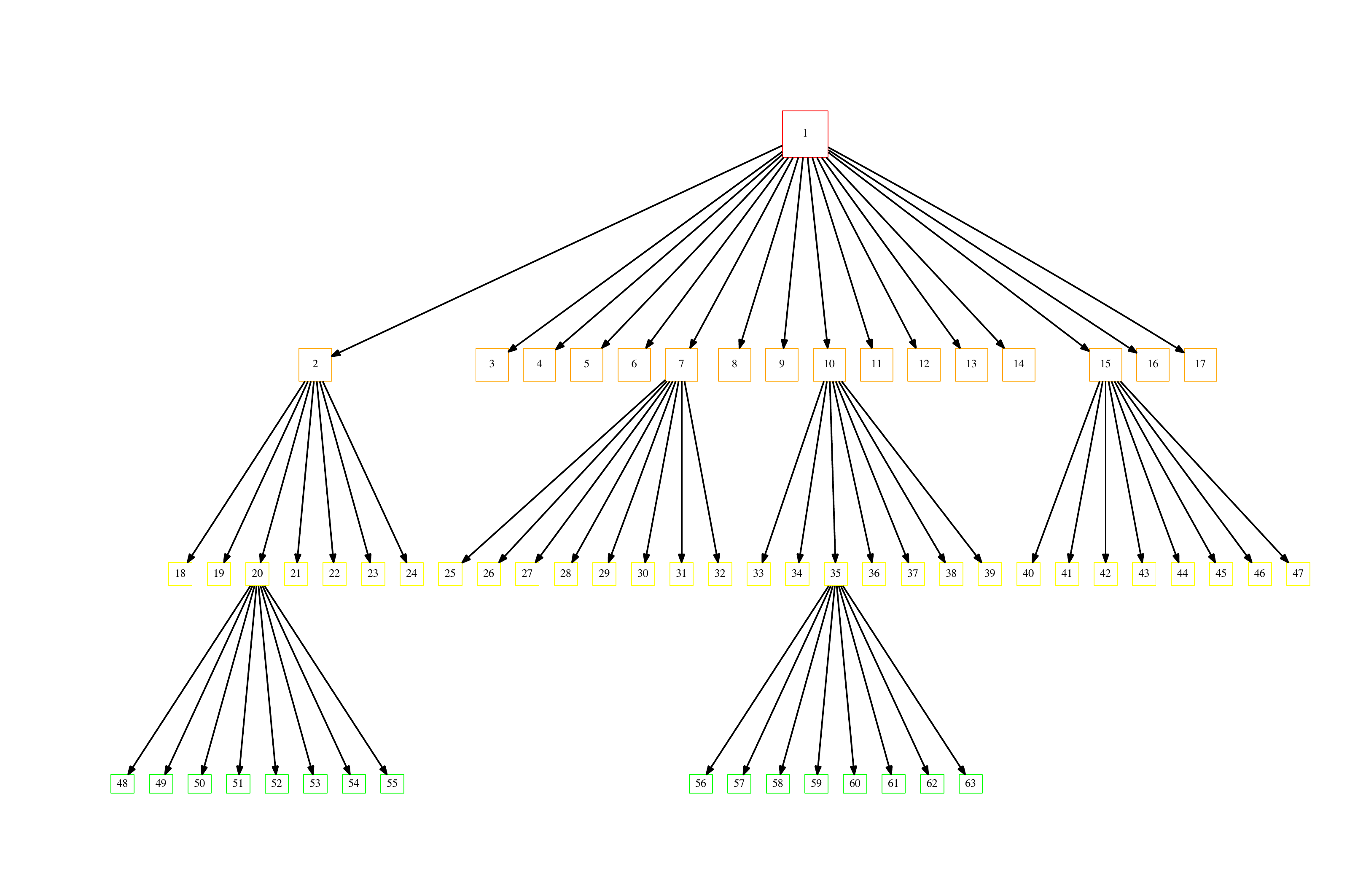}
\caption{Possible images of $\tau_{E,2^{k}}$}
\end{figure}
\end{center}
   
The subgroup $1$ corresponds to $\ghappy$. The sixteen orange boxes on
the second level represent index 2 subgroups of $\ghappy$, while those
boxes on the third and fourth level represent index 4 and index 8
subgroups, respectively.

Suppose that $E/\Q$ is an elliptic curve with a rational point of
order $2$ and $\alpha \in E(\Q)$ is a point of infinite order. Assume
that that $E : y^{2} = x^{3} + ax^{2} + bx$. We denote $\alpha =
(c,ck)$, which forces $b = ck^{2} - ac - c^{2}$. We will specify
curves in this family by the parameters $a$, $c$ and $k$. The group
$\ghappy$ has $31$ maximal subgroups of index $2$. These subgroups
can be generated (using the $*$ operation of Section~\ref{grouptheory})
from five subgroups. One of these is
\[
M = \left\{ \begin{bmatrix} a & b & 0 \\ c & d & 0 \\ e & f & 1 \end{bmatrix} : \begin{bmatrix} a & b \\ c & d \end{bmatrix} \in \Gamma_{0}(2) \text{ and }
e \equiv 0 \pmod{2} \right\}.
\]
The other four are maximal subgroups of the form
\[
\left\{ \begin{bmatrix} a & b & 0 \\ c & d & 0 \\ e & f & 1 \end{bmatrix} :
\begin{bmatrix} a & b \\ c & d \end{bmatrix} \in K \right\}
\]
where $K \subseteq \Gamma_{0}(2)$ is a maximal subgroup. We define $M_{-1}$
to be the maximal subgroup of this form corresponding to
$K = \{ g \in \Gamma_{0}(2) : \det(g) \equiv 1 \pmod{4} \}$, $M_{2}$
to be the maximal subgroup of this form corresponding to
$K = \{ g \in \Gamma_{0}(2) : \det(g) \equiv \pm 1 \pmod{8} \}$,
$M_{b}$ to correspond to $K = \left\{ \begin{bmatrix} a & b \\ c & d \end{bmatrix} \in \Gamma_{0}(2) : c \equiv 0 \pmod{4} \right\}$ and
$M_{a^{2}-4b}$ to correspond to $K = \left\{ \begin{bmatrix} a & b \\ c & d \end{bmatrix} \in \Gamma_{0}(2) : b \equiv 0 \pmod{2} \right\}$. For $d \in \{ -1,2,b,a^{2}-4b \}$ the fixed field of $\tau_{E,2^{k}}^{-1}(M_{d})$ is $\Q(\sqrt{d})$.

Let $H_{i}$ denote the $i$th subgroup in the chart above. (For a Magma
file that contains generators of each subgroup, as well as other
scripts and log files related to the computations in this section, see
the page \url{http://users.wfu.edu/rouseja/liang-rouse/}.)

The subgroup $H_{2}$ equals the subgroup $M$
defined above. We have that $\im \tau \subseteq H_{2}$ if and only if
for all $\sigma \in \Gal(K_{k}/\Q)$, $\sigma(\beta_{1}) = \beta_{1}$
or $\beta_{1} + T$. This implies that $\alpha' = \phi(\beta_{1}) =
\phi(\beta_{1} + T) \in E'(\Q)$ is rational. Since $\psi(\alpha') =
\alpha$, this implies that the $x$-coordinate of $\alpha$, $c$, is a
perfect square. Hence, $\im \tau \subseteq H_{2}$ if and only if $c$
is a perfect square. This gives the interpretation for subgroup $2$.
It follows from this that $\im \tau = \ghappy$ if and only if
$\langle -1, 2, b, a^{2} - 4b \rangle \subseteq (\Q^{\times})/(\Q^{\times})^{2}$
has order $32$ (which is equivalent to there being no multiplicative
relations, mod squares, involving $-1$, $2$, $b$, $a^{2}-4b$ and $c$).

Given a number $d \in \langle -1, 2, b, a^{2} - 4b \rangle$, we let
$M_{d} \subseteq \ghappy$ denote the maximal subgroup with the property
that $g \in M_{d}$ if and only if when $\tau_{E,2^{k}}(\sigma) = g$,
we have that $\sigma$ fixes $\sqrt{d}$. That is, $M_{d}$ is the subgroup of
$\ghappy$ corresponding to $\Q(\sqrt{d})$. It is not hard to see that
$\im \tau_{E,2^{k}} \subseteq H_{2} * M_{d}$ if and only if $cd$ is a perfect
square, and this furnishes the interpretations for subgroups $3$ through $17$.
We will use this trick for other batches of subgroups (18-24, 25-32, 33-39, 40-47, 48-55, and 56-63). In particular, we will compute the interpretation
for one such subgroup, and then use this group-theoretic method to find
the interpretation for all other subgroups in that batch.

It follows from the intrepretations above that maximal subgroups
of subgroups 2 and 10 arise because either the point $\alpha$ or $\alpha+T$
is the image of a rational point under the isogeny $\psi : E' \to E$. The
situation where subgroups 7 and 15 is a bit more mysterious, and we will
take a moment to explain where those maximal subgroups come from. First,
the discriminant of the curve $E$ is $\Delta(E) = 16 b^{2} (a^{2} - 4b)$, and an examination
of the classification of $2$-adic images of elliptic curves over $\Q$ by the second author and David Zureick-Brown (see \cite{RZB}) shows that for any
curve $E/\Q$, $\sqrt[4]{\Delta(E)} \in \Q(E[4])$. (In fact, the modular curve $X_{2a}$ of \cite{RZB} parametrizes elliptic curves whose discriminant is a fourth power.) Now, the image of $\tau$ is contained in subgroup $15$ if and only if $\frac{-(a^{2}-4b)}{c} = d^{2}$ for some $d \in \Q$, and in this case
the field $\Q(\beta_{1})$ obtained by adjoining a preimage of $\alpha$ under multiplication by $2$ is also a radical extension: we have $\Q(\beta_{1})
= \Q(\sqrt[4]{c} \sqrt{2cd + 4kc})$. This implies that the quantity
\[
  \frac{-(a^{2} - 4b)}{c (2cd+4kc)^{2}}
\]
is a square in $\Q$, but is a fourth power in $\Q(\beta_{1},E[8])$, and this means that
\[
\Q\left(\sqrt[4]{ \frac{-(a^{2} - 4b)}{c (2cd+4kc)^{2}}}\right)/\Q
\]
gives rise to a quadratic extension which examples show can be independent
of the other $15$ quadratic subextensions coming from square roots of
products of $\{-1,2,b,a^{2}-4b\}$. This explains the presence of the maximal
subgroups of subgroup $15$. (For the maximal subgroups of group $7$,
one can consider instead the field obtained by taking a preimage
of $\alpha+T$.)

To compute the interpretation for one subgroup in each of our batches,
we will use a technique from \cite{DokDok} (see the proof of the Lemma
on page 962) and \cite{2015REU} (see the proof of Lemma 9.1). Given a
subgroup $H \subseteq \AGL_{2}(\Z/2^{k} \Z)$ and an elliptic curve $E$
(depending on the parameters $a$, $c$ and $k$), we wish to compute a
polynomial $f_{a,c,k}(x) \in \Z[a,c,k][x]$ so that $f_{a,c,k}(x)$ has
a rational root if and only if the image of $\tau_{E,2^{k}}$ is
contained in $H$. Let $\beta_{k}$ denote a point so that $2^{k}
\beta_{k} = \alpha$, and let $A$ and $B$ be points which generate
$E[2^{k}] \cong (\Z/2^{k} \Z)^{2}$. For each right coset $C$ of $H$ in
$\AGL_{2}(\Z/2^{k} \Z)$, let
\[
  \zeta_{C} = \sum_{\sigma \in C} \sigma(x(\beta_{2} + A)) \sigma(x(\beta_{2} + B)) \sigma(x(\beta_{2} + A + B))^{2}.
\]
Then, let $f_{a,c,k}(x) = \prod_{C} (x - \zeta_{C})$. A straightforward computation shows that the $x$-coordinates of the points $\beta_{k} + r_{1} A + r_{2} B$,
$0 \leq r_{1}, r_{2} < 2^{k}$ are integral over $\Z[a,c,k]$ and it follows
that $f_{a,c,k}(x) \in \Z[a,c,k][x]$. Moreover, the change of variables
$(x,y) \mapsto (t^{2} x, t^{3} y)$ sends $a \mapsto t^{2} a$,
$c \mapsto t^{2} c$ and $k \mapsto tk$. For this reason, we say that
$a$ and $c$ have weight $2$ and $k$ has weight $1$. The map
$(x,y) \mapsto (t^{2} x, t^{3} y)$ sends $\zeta_{C}$ to $t^{8} \zeta_{C}$,
and it follows from this that the weight of the coefficient of $x^{i}$
in $f_{a,c,k}$ is $8 \deg(f_{a,c,k}) - 8i$. To find the polynomial
$f_{a,c,k}$ one then simply computes $f_{a,c,k}$ for several triples $(a,c,k)$
(by numerically computing the $x$-coordinates of the $\beta_{k} + r_{1} A + r_{2} B$ using the fact that $E(\C) \cong \C/\Lambda$ for some lattice $\Lambda$),
and uses linear algebra to determine the coefficients of each monomial
of the correct weight in the coefficient of $x^{i}$. This method produces
the polynomial $f_{a,c,k}(x) \in \Z[a,c,k][x]$.

Next, we present tables describing the various images. The table below
gives for each subgroup an element with the property that $\im
\tau_{E,2^{k}}$ is contained in that subgroup if and only if that
element is a square in $\Q$, a curve (specified by $[a,c,k]$) whose
image is that subgroup, the density of odd order reductions associated
with that image of $\tau_{E,2^{k}}$ (for more detail about how these
were computed, see Section~\ref{calc_density}), and the value of
$\pi_{S}(x)/\pi(x)$ for the curve specified with $x = 10^{7}$.

\begin{tabular}{c|cccc}
  Subgroup & Element & Curve & Density & $\pi_{S}(10^{7})/\pi(10^{7})$\\
  \hline
  $1$ & N/A & $[3,3,1]$ & $5/21 \approx 0.238095$ & $0.237796$\\
  $2$ & $c$ & $[-3,1,3]$ & $10/21 \approx 0.476191$ & $0.476216$\\
  $3$ & $-2bc (a^{2} - 4b)$ & $[14,15,6]$ & $5123/21504 \approx 0.238235$ & $0.238055$\\
  $4$ & $-2bc$ & $[3,3,2]$ & $5123/21504 \approx 0.238235$ & $0.238167$\\
  $5$ & $2bc (a^{2}-4b)$ & $[6,-5,2]$ & $5123/21504 \approx 0.238235$ & $0.237885$\\
  $6$ & $2bc$ & $[2,-3,1]$ & $5123/21504 \approx 0.238235$ & $0.238034$\\
  $7$ & $-bc (a^{2} - 4b)$ & $[-5,11,1]$ & $83/336 \approx 0.247024$ & $0.246844$\\
  $8$ & $-bc$ & $[2,3,1]$ & $83/336 \approx 0.247024$ & $0.247027$\\
  $9$ & $bc(a^{2}-4b)$ & $[10,-11,2]$ & $13/42 \approx 0.309524$ & $0.309546$\\
  $10$ & $bc$ & $[2,-5,1]$ & $1/7 \approx 0.142857$ & $0.142902$\\
  $11$ & $-2c(a^{2}-4b)$ & $[14,7,6]$ & $5123/21504 \approx 0.238235$ & $0.238993$\\
  $12$ & $-2c$ & $[1,-2,2]$ & $5123/21504 \approx 0.238235$ & $0.237946$\\
  $13$ & $2c(a^{2}-4b)$ & $[6,3,2]$ & $5123/21504 \approx 0.238235$ & $0.238411$\\
  $14$ & $2c$ & $[2,2,1]$ & $5123/21504 \approx 0.238235$ & $0.238438$\\
  $15$ & $-c(a^{2} - 4b)$ & $[-5,-5,1]$ & $83/336 \approx 0.247024$ & $0.247187$\\
  $16$ & $-c$ & $[-1,-1,1]$ & $83/336 \approx 0.247024$ & $0.246790$\\
  $17$ & $c(a^{2}-4b)$ & $[10,5,2]$ & $1/7 \approx 0.142857$ & $0.143135$\\
\end{tabular}

Subgroup $18-24$ arise as images when $c = s^{2}$ is a perfect square.
In this case, there are two points $\alpha' \in E'(\Q)$
so that $\psi(\alpha') = \alpha$. We have
\[
\alpha' = (a+2s^{2}+2sk, 2s(a + 2s^{2} + 2sk))
\]
and the two choices of $\alpha'$ correspond to the choice of the sign of $s$.
In the table below, we specify the square class of $x(\alpha')$
and this means that the image of $\tau$ is contained in the corresponding
subgroup when \emph{either} choice of $\alpha'$ results in $x(\alpha')$ being
in the correct square class.

\begin{tabular}{c|cccc}
  Subgroup & Element & Curve & Density & $\pi_{S}(10^{7})/\pi(10^{7})$\\
  \hline
  $18$ & $-2b x(\alpha')$ & $[7,16,3]$ & $5123/10752 \approx 0.476470$ & $0.477010$\\
  $19$ & $2b x(\alpha')$ & $[-3,1,2]$ & $5123/10752 \approx 0.476470$ & $0.476296$\\
  $20$ & $-b x(\alpha')$ & $[28,36,1]$ & $83/168 \approx 0.494048$ & $0.494378$\\
  $21$ & $b x(\alpha')$ & $[-5,1,4]$ & $19/42 \approx 0.452381$ & $0.452681$\\
  $22$ & $-2 x(\alpha')$ & $[2,1,3]$ & $5123/10752 \approx 0.476470$ & $0.476195$\\
  $23$ & $2 x(\alpha')$ & $[-4,1,2]$ & $5123/10752 \approx 0.476470$ & $0.476213$\\
  $24$ & $-x(\alpha')$ & $[3,1,3]$ & $83/168 \approx 0.494048$ & $0.493984$\\
\end{tabular}

Now we consider the batch of subgroups $25-32$ which are subgroups of
group $7$.  The image of $\tau$ is contained in group $7$ if and only
if $-bc (a^{2}-4b)$ is a square, which is equivalent to $-(a^{2} - 4b)
x(\alpha+T)$ being a square.  We have $x(\alpha+T) = -a -c + k^{2}$
and so we let $-(a^{2} - 4b)(-a-c+k^{2}) = d^{2}$. As above, there are
two choices of $d$. Note that the quantities listed in the Element
table can sometimes equal zero, and in this case, the table below does
not determine the image of $\tau$.

\begin{tabular}{c|cccc}
  Subgroup & Element & Curve & Density & $\pi_{S}(10^{7})/\pi(10^{7})$\\
  \hline
  $25$ & $2d(d+2ak + 2ck - 2k^{3})$ & $[-6,13,2]$ & $2659/10752 \approx 0.247303$ & $0.247257$\\
  $26$ & $-d(d+2ak + 2ck - 2k^{3})$ & $[6,7,4]$ & $89/336 \approx 0.264881$ & $0.265085$\\
  $27$ & $-2d(d + 2ak + 2ck - 2k^{3})$ & $[10,21,6]$ & $2659/10752 \approx 0.247303$ & $0.247638$\\
  $28$ & $d(d + 2ak + 2ck - 2k^{3})$ & $[28,-12,3]$ & $25/112 \approx 0.223214$ & $0.223038$\\
  $29$ & $2bd(d+2ak+2ck-2k^{3})$ & $[15,6,6]$ & $2659/10752 \approx 0.247303$ & $0.247663$\\
  $30$ & $-bd(d + 2ak + 2ck - 2k^{3})$ & $[-210,375,9]$ & $25/112 \approx 0.223214$ & $0.223066$\\
  $31$ & $-2bd(d+2ak+2ck-2k^{3})$ & $[30,10,10]$ & $2659/10752 \approx 0.247303$ & $0.247015$\\
  $32$ & $bd(d+2ak+2ck-2k^{3})$ & $[-55,125,9]$ & $89/336 \approx 0.264881$ & $0.265033$\\
\end{tabular}

Now we consider the batch of subgroups 33-39, which are subgroups of group 10.
This implies that $bc$ is a square, which implies that $x(\alpha+T)$ is
a square and hence there are two points $\alpha' \in E'(\Q)$
with $\psi(\alpha') = \alpha+T$. The table is the following.

\begin{tabular}{c|cccc}
  Subgroup & Element & Curve & Density & $\pi_{S}(10^{7})/\pi(10^{7})$\\
  \hline
  $33$ & $-2bx(\alpha')$ & $[-7,-14,2]$ & $513/3584 \approx 0.143136$ & $0.143191$\\
  $34$ & $2b x(\alpha')$ & $[-3,6,2]$ & $513/3584 \approx 0.143136$ & $0.143121$\\
  $35$ & $-b x(\alpha')$ & $[-40,45,3]$ & $5/42 \approx 0.119048$ & $0.119048$\\
  $36$ & $b x(\alpha')$ & $[10,10,6]$ & $5/42 \approx 0.119048$ & $0.118733$\\
  $37$ & $-2 x(\alpha')$ & $[2,3,3]$ & $513/3584 \approx 0.143136$ & $0.143154$\\
  $38$ & $2 x(\alpha')$ & $[-2,7,3]$ & $513/3584 \approx 0.143136$ & $0.143036$\\
  $39$ & $-x(\alpha')$ & $[2,5,4]$ & $5/42 \approx 0.119048$ & $0.118566$\\
\end{tabular}

We consider the batch of subgroups 40-47, which are subgroups of group 15.
We write $-c(a^{2}-4b) = d^{2}$, and there are two choices for $d$. As before,
if either of these choices for $d$ makes the quantity listed in the Element
column of the table a non-zero square, the image of $\tau$ is contained
in the corresponding subgroup.

\begin{tabular}{c|cccc}
  Subgroup & Element & Curve & Density & $\pi_{S}(10^{7})/\pi(10^{7})$\\
  \hline
  $40$ & $2bcd(d+2ck)$ & $[-45,60,5]$ & $2659/10752 \approx 0.247303$ & $0.247000$\\
  $41$ & $bcd(d+2ck)$ & $[-210,-21,12]$ & $89/336 \approx 0.264881$ & $0.265084$\\
  $42$ & $-2bcd(d+2ck)$ & $[15,15,6]$ & $2659/10752 \approx 0.247303$ & $0.247006$\\
  $43$ & $-bcd(d+2ck)$ & $[-55,11,-9]$ & $89/336 \approx 0.264881$ & $0.264981$\\
  $44$ & $2cd(d+2ck)$ & $[10,5,6]$ & $2659/10752 \approx 0.247303$ & $0.247056$\\
  $45$ & $cd(d+2ck)$ & $[6,3,4]$ & $25/112 \approx 0.223214$ & $0.223415$\\
  $46$ & $-2cd(d+2ck)$ & $[-6,-3,2]$ & $2659/10752 \approx 0.247303$ & $0.247212$\\
  $47$ & $-cd(d+2ck)$ & $[-14,-7,6]$ & $25/112 \approx 0.223214$ & $0.223007$\\
\end{tabular}

For the batch of subgroups 48-55, we have the same sort of phenomenon
happening as with the maximal subgroups of groups 7 and 10, except this
time on $E'$ instead of $E$. The image of $\tau$ is contained in group 20
if $-b x(\alpha')$ is a square, which is the same as saying
that $-(b/c) (a + 2s^{2} - 2sk)$ is a square (where again $c = s^{2}$).
Write $-(b/c) (a + 2s^{2} - 2sk) = d^{2}$ and recall now that there are two
choices for $s$ and two choices for $d$ (per $s$ with $-b (a + 2s^{2} - 2sk)$ a square). If \emph{any} of these choices make it so the quantity in the
Element column is a square, then the image of $\tau$ is contained in the
corresponding subgroup.

\begin{tabular}{c|cccc}
  Subgroup & Element & Curve & Density & $\pi_{S}(10^{7})/\pi(10^{7})$\\
  \hline
  $48$ & $2(a^{2}-4b)(d + x(\alpha+T))$ & $[60,36,9]$ & $2659/5376 \approx 0.494606$ & $0.494232$\\
  $49$ & $b(a^{2}-4b)(d + x(\alpha+T))$ & $[30,121,1]$ & $89/168 \approx 0.529762$ & $0.529399$\\
  $50$ & $-2 (a^{2}-4b)(d + x(\alpha+T))$ & $[90,16,16]$ & $2659/5376 \approx 0.494606$ & $0.494423$\\
  $51$ & $-2b (a^{2}-4b)(d + x(\alpha+T))$ & $[210,81,21]$ & $41/84 \approx 0.488095$ & $0.487864$\\
  $52$ & $-2b (d + x(\alpha+T))$ & $[15,9,2]$ & $2659/5376 \approx 0.494606$ & $0.494824$\\
  $53$ & $-(d+x(\alpha+T))$ & $[-12,16,1]$ & $41/84 \approx 0.488095$ & $0.488500$\\
  $54$ & $2b (d + x(\alpha+T))$ & $[3,1,4]$ & $2659/5376 \approx 0.494606$ & $0.494469$\\
  $55$ & $d + x(\alpha+T)$ & $[-7,16,11]$ & $25/56 \approx 0.446429$ & $0.446331$\\
\end{tabular}

The last batch of subgroups is 56-63, which are subgroups of group 35. For
these groups, $x(\alpha+T) = -a-c+k^{2}$ is a square, so write
$-a-c+k^{2} = s^{2}$. The image of $\tau$ is contained in group $35$ if
and only if $-b x(\alpha')$ is a square, which means
$(a + s^{2} - k^{2})(a + 2s^{2} - 2sk) = d^{2}$. (Again, there are two
choices for $s$, and for a valid choice of $s$, two choices for $d$.)
The table for this batch is the following.

\begin{tabular}{c|cccc}
  Subgroup & Element & Curve & Density & $\pi_{S}(10^{7})/\pi(10^{7})$\\
  \hline
  $56$ & $d + c$ & $[7,112,12]$ & $5/21 \approx 0.238095$ & $0.237694$\\
  $57$ & $2(a^{2}-4b)(d+c)$ & $[-30,-15,6]$ & $643/5376 \approx 0.119606$ & $0.119463$\\
  $58$ & $(a^{2}-4b)(d+c)$ & $[30,-150,1]$ & $1/14 \approx 0.0714285$ & $0.0715225$\\
  $59$ & $-2(a^{2}-4b)(d+c)$ & $[-60,-15,5]$ & $643/5376 \approx 0.119606$ & $0.119643$\\
  $60$ & $-(a^{2}-4b)(d+c)$ & $[210,150,21]$ & $19/168 \approx 0.113095$ & $0.112867$\\
  $61$ & $-2(d+c)$ & $[5,-20,1]$ & $643/5376 \approx 0.119606$ & $0.119774$\\
  $62$ & $-(d+c)$ & $[-12,-3,1]$ & $19/168 \approx 0.113095$ & $0.113223$\\
  $63$ & $2(d+c)$ & $[3,12,4]$ & $643/5376 \approx 0.119606$ & $0.120063$\\
\end{tabular}

\section{Calculating the density}
\label{calc_density}

In this section, we give a convenient way to compute the limit
on the right hand side of \eqref{densityeq},
and in this section we allow $\ell$ to be any prime number and
$G \subseteq \AGL_{2}(\Z/\ell^{r} \Z)$ to be a subgroup of index $m$.
The right hand side of \eqref{densityeq} is
\[
\mathcal{F}(G) = \lim_{k \to \infty} \frac{\# \{ (\vec{v},g) \in \AGL_{2}(\Z/\ell^{k} \Z) : (\vec{v} \bmod \ell^{r}, g \bmod \ell^{r}) \in G \text{ and } \vec{v} \in \Row(g-I) \}}{\# \{ (\vec{v},g) \in \AGL_{2}(\Z/\ell^{k} \Z) : (\vec{v} \bmod \ell^{r}, g \bmod \ell^{r}) \in G \}}.
\]
By Theorem $3.2$ of \cite{Jones} if $E/\Q$ is an elliptic curve
without complex multiplication, $\alpha \in E(\Q)$, and the image of
$\tau_{E,\ell^{k}}$ is the full preimage in $\AGL_{2}(\Z/\ell^{k} \Z)$
of $G$ for all $k$, the density of primes $p$ for which $\alpha \in
E(\F_{p})$ has order coprime to $\ell$ is equal to $\mathcal{F}(G)$.

A general procedure for computing $\mathcal{F}(G)$ with a finite
amount of computation has already been given by Lombardo and Perucca
(see \cite{LombardoPerucca} and \cite{LPNYJM}). They also handle the
cases that arise for elliptic curves with complex multiplication.  We
wish to give a few simple rules that allow for easy computer
computation of $\mathcal{F}(G)$. Given that our results essentially follow
from the work of Lombardo and Perucca, we will not give full details in
the proofs.

For $\vec{v} \in (\Z/\ell^{r} \Z)^{2}$ and $M \in M_{2}(\Z/\ell^{r} \Z)$,
we define
\small
\begin{equation*}
  \mu_{r}(\vec{v},M) = \lim_{k \to \infty}
  \frac{\#\{ (\tilde{v},\tilde{M}) \in (\Z/\ell^{k} \Z)^{2} \times
M_{2}(\Z/\ell^{k} \Z) : (\tilde{v},\tilde{M}) \equiv (\vec{v},M) \pmod{\ell^{r}},
\tilde{v} \in \Row(\tilde{M}-I) \}}{(\# \AGL_{2}(\Z/\ell^{r} \Z)/m) \cdot \ell^{6(k-r)}}.
\end{equation*}
\normalsize
The quantity $\mu_{r}(\vec{v},M)$ is the contribution to the density
of all lifts of $(\vec{v},M) \in G$ and so $\mathcal{F}(G) = \sum_{(\vec{v},M) \in G} \mu_{r}(\vec{v},g)$.

\begin{lem}
\label{lem1}
If $\vec{v} \not\in \Row(M-I)$, then $\mu_{r}(\vec{v},M) = 0$.
\end{lem}
\begin{proof}
It is easy to see that no lifts $(\tilde{v},\tilde{M})$ of $(\vec{v},M)$
have $\tilde{v} \in \Row(\tilde{M}-I)$ either.
\end{proof}

\begin{lem}
\label{lem2}
If $\vec{v} \equiv 0 \pmod{\ell}$ and $M-I \equiv 0 \pmod{\ell}$,
then
\[
\mu_{r}(\vec{v},M) = \frac{1}{\ell^{6}} \mu_{r-1}\left(\frac{\vec{v}}{\ell},
  \left(\frac{(M-I)}{\ell}\right) + I\right).
\]
\end{lem}
\begin{proof}
  This follows easily from the fact that a solution to $\tilde{v} \equiv \tilde{x} (\tilde{M}-I) \pmod{\ell^{k}}$ is equivalent to a solution
  to $\frac{\tilde{v}}{\ell} \equiv \tilde{x} \frac{M-I}{\ell} \pmod{\ell^{k-1}}$. 
\end{proof}

\begin{lem}
\label{lem3}
If $\det(M-I) \not\equiv 0 \pmod{\ell^{r}}$ and $\vec{v} \in \Row(M-I)$, then
\[
\mu_{r}(\vec{v},M) = \frac{m}{\# \AGL_{2}(\Z/\ell^{r} \Z)}.
\]
\end{lem}
\begin{proof}
  This lemma is equivalent to the statement that every vector $\tilde{v} \equiv \vec{v} \pmod{\ell^{r}}$ is in $\Row(\tilde{M}-I)$ for any matrix $\tilde{M} \equiv M \pmod{\ell^{r}}$. If $y$ is an integer so that $y \det(\tilde{M}-I) \equiv \ell^{r-1} \pmod{\ell^{r}}$, this follows from the formula $y {\rm adj}(\tilde{M}-I) (\tilde{M}-I) \equiv \ell^{r-1} I \pmod{\ell^{r}}$. Here ${\rm adj}(\tilde{M}-I)$ is the adjoint
  of $\tilde{M}-I$.
\end{proof}

\begin{lem}
\label{lem4}
Suppose that $\det(M-I) \equiv 0 \pmod{\ell^{r}}$, $\vec{v} \in \Row(M-I)$,
but not all entries in $M-I$ are $\equiv 0 \pmod{\ell}$. Then
\[
  \mu_{r}(\vec{v},M) = \frac{m}{\ell^{6r-4} (\ell-1)^{2} (\ell+1)^{2}}.
\]
\end{lem}
\begin{proof}
We let $\vec{v} = \begin{bmatrix} \epsilon & \zeta \end{bmatrix}$ and
$M = I + \begin{bmatrix} \alpha & \beta \\ \gamma & \delta \end{bmatrix}$.
We consider lifts $(\tilde{v},\tilde{M})$ mod $\ell^{r+1}$ with
$\tilde{v} = \begin{bmatrix} \epsilon + e \ell^{r} & \zeta + f \ell^{r} \end{bmatrix}$, $\tilde{M} = M + \ell^{r} \begin{bmatrix} a & b \\ c & d \end{bmatrix}$
where $a, b, c, d, e, f \in \Z/\ell \Z$.

One can see that $\det(\tilde{M}-I) \equiv \det(M) + (\alpha \delta - b \gamma - c \beta + d \alpha) \ell^{r} \pmod{\ell^{r+1}}$ and the assumption that
$M \not\equiv I \pmod{\ell}$ guarantees that there are $\ell^{6}-\ell^{5}$ lifts
$(\tilde{v},\tilde{M})$ of $(v,M)$ with $\det(\tilde{M}-I) \not\equiv 0 \pmod{\ell^{r+1}}$. It then suffices to compute the number of lifts $(\tilde{v},\tilde{M})$ with $\det(\tilde{M}-I) \equiv 0 \pmod{\ell^{r+1}}$ and
$\tilde{v} \in \Row(\tilde{M}-I)$. Cramer's rule shows that this occurs if
and only if the choice of $(a,b,c,d,e,f)$ satisfies
\begin{align*}
  a \delta - \gamma b - \beta c + \alpha d &\equiv -\frac{\alpha \delta - \beta \gamma}{\ell^{r}} \pmod{\ell}\\
  -c \zeta + d \epsilon + e \delta - f \gamma &\equiv -\frac{\delta \epsilon - \gamma \zeta}{\ell^{r}} \pmod{\ell}\\
  a \zeta - b \epsilon - \beta e + \alpha f &\equiv -\frac{\alpha \zeta - \beta \epsilon}{\ell^{r}} \pmod{\ell}.
\end{align*}
If we assume that $\alpha \not\equiv 0 \pmod{\ell}$, one can choose arbitrary values for $a, b, c$ and $e$ and then solve the first and third congruences for $d$ and $f$. With these values for $d$ and $f$, one can prove the second congruence is also true. (So the second congruence follows from the first and the third, assuming that $\alpha \not\equiv 0 \pmod{\ell}$.) This implies that there are $\ell^{4}$ lifts $(\tilde{v},\tilde{M})$ with $\tilde{v} \in \Row(\tilde{M}-I)$. (Similarly arguments apply if
$\beta$, $\gamma$ or $\delta$ is $\not\equiv 0 \pmod{\ell}$.) This implies that
\[
 \mu_{r}(\vec{v},M) = \frac{m \left(1 - \frac{1}{\ell}\right)}{\# \AGL_{2}(\Z/\ell^{r} \Z))} + \sum_{(\tilde{v},\tilde{M})} \mu_{r+1}(\tilde{v},\tilde{M}), 
\]
where the sum is over the pairs $(\tilde{v},\tilde{M})$ with $\tilde{v} \in \Row(M-I)$ and $\det(\tilde{M}) \equiv 0 \pmod{\ell^{r+1}}$. We apply the above
equation repeatedly and obtain that
\[
\mu_{r}(\vec{v},M) = \frac{m}{\# \AGL_{2}(\Z/\ell^{r} \Z)} \cdot
\left[1 + \frac{1}{\ell^{2}} + \frac{1}{\ell^{4}} + \cdots\right]
= \frac{m}{(\ell-1)^{2} (\ell+1)^{2} \ell^{6r-4}}.
\]
\end{proof}

If we have a pair $(\vec{v},M)$ with $\det(M-I) \not\equiv 0 \pmod{\ell^{r}}$,
then Lemma~\ref{lem3} applies. If $\det(M-I) \equiv 0 \pmod{\ell^{r}}$ but
not all entries in $M-I$ are $\equiv 0 \pmod{\ell}$, then Lemma~\ref{lem4}
applies. If $M-I \equiv 0 \pmod{\ell}$, then we can apply Lemma~\ref{lem2}
repeatedly until $M-I \not\equiv 0 \pmod{\ell}$, except for the
case that $(\vec{v},M) = (\vec{0},I)$.

\begin{lem}
\label{lem5}
We have $\mu_{r}(\vec{0},I) = \frac{((\ell-1)\ell + 1)m}{(\ell-1) \ell^{6r-5} (\ell^{6} - 1)}$.
\end{lem}
\begin{proof}
  By using Lemma~\ref{lem2}, it suffices to handle the case that $r = 1$. Lemma~\ref{lem2} also gives that
\[
\mu_{1}(\vec{0},I) = \frac{1}{\ell^{6}} \mu_{0}(\vec{0},I) = \frac{1}{\ell^{6}}
\sum_{\substack{\vec{v} \in (\Z/\ell \Z)^{2}\\ M \in M_{2}(\Z/\ell \Z)}} \mu_{1}(\vec{v},M).
\]
There are $\ell^{2} \cdot \# \GL_{2}(\Z/\ell \Z)$ pairs $(\vec{v},M)$
for which $M-I$ is invertible. The contribution of these pairs
is $\frac{m}{\ell^{6}}$ by Lemma~\ref{lem3}.

If $\det(M-I) \equiv 0 \pmod{\ell}$ but $M-I$ is not the zero matrix,
then $\# \Row(M-I) = \ell$ and there are $\ell^{3} + \ell^{2} - \ell - 1$
matrices that fall into this case. By Lemma~\ref{lem4}, the contribution
of these cases is
\[
\frac{1}{\ell^{6}} \cdot \left(\frac{m}{(\ell-1)^{2} \ell^{2} (\ell+1)^{2}}\right) \cdot \ell (\ell^{3} + \ell^{2} - \ell - 1) = \frac{m}{(\ell - 1) \ell^{7}}.
\]

Finally, we have the case that $\vec{v} = \vec{0}$ and $M = I$ and we get
\[
\mu_{1}(\vec{0},I) = \frac{m}{\ell^{6}} + \frac{m}{(\ell-1) \ell^{7}} +
\frac{1}{\ell^{6}} \mu_{1}(\vec{0},I).
\]
Solving for $\mu_{1}(\vec{0},I)$ gives the desired result.
\end{proof}

\bibliographystyle{plain}
\bibliography{refs}

\end{document}